\documentclass[11pt,a4paper]{article}
\pdfoutput=1

\usepackage{epsf,epsfig,amsfonts,amsgen,amsmath,amstext,amsbsy,amsopn,amsthm,cases,listings,color
}
\usepackage{ebezier,eepic}
\usepackage{color}
\usepackage{multirow}
\usepackage{epstopdf}
\usepackage{graphicx}
\usepackage{pgf,tikz}
\usepackage{mathrsfs}
\usepackage[marginal]{footmisc}
\usepackage{enumitem}
\usepackage[titletoc]{appendix}
\usepackage{booktabs}
\usepackage{url}
\usepackage{relsize}
\usepackage{mathtools}
\usepackage{comment}
\usepackage{pgfplots}
\usepackage{amssymb}
\usepackage{float}
\usepackage{wasysym}

\usepackage{empheq}
\usepackage{ifthen}

\usepackage{dsfont}

\usepackage{tikz}
\usepackage{longtable}
\usepackage{subfigure}
\pgfplotsset{compat=1.17}
\usepackage{mathrsfs}
\usepackage{wasysym}
\usetikzlibrary{arrows}
\usepackage{aligned-overset}
\usepackage{bm}
\usepackage{bbm}
\usepackage[T1]{fontenc}
\usepackage[nobysame,initials,bibtex-style]{amsrefs}

\allowdisplaybreaks[1]

\usepackage{color}

\usepackage{listings}

\usepackage{hyperref}

\newtheorem{definition}{Definition} [section]
\newtheorem{theorem}[definition]{Theorem}
\newtheorem{lemma}[definition]{Lemma}

\theoremstyle{remark}
\newtheorem{remark}[definition]{Remark}

\newcommand{\C}[1]{\mathcal{#1}}

\def\FF{\mathbb{F}}

\newcommand{\cG}[2]{{\mathcal G}^{(#1)}_{#2}}
\renewcommand{\O}[1]{\overline{#1}}
\newcommand{\p}[2]{P_{#1}(#2)}
\newcommand{\pp}[3]{P_{\,\O{#1}#2}(#3)}
\newcommand{\CI}{C}
\newcommand{\I}[1]{{\mathbbm #1}}

\newcommand{\OutIn}[2]{$\ifthenelse{\equal{#1}{}}{}{\O #1}#2$}

\newcommand{\OMerge}[2]{\ifthenelse{\equal{#1}{}}{(#2)}{(#1|#2)}}
\newcommand{\Merge}[2]{\ifthenelse{\equal{#1}{}}{#2}{#1|#2}}

\def\ex{\mathrm{ex}}

\newcommand{\hide}[1]{}

\newcommand{\bes}[1]{\pi(#1)} 

\begin{document}
\title{\bf\Large The quadratic Brown--Erd\H{o}s--S\'os problem for 3-uniform hypergraphs with 8 and 9 edges}
\date{}
\author{
Levente Bodn\'ar\footnote{The work was carried out while the author was affiliated with the University of Warwick. Email: \mbox{bodnalev@gmail.com}
}
\and
Oleg Pikhurko\footnote{Mathematics Institute and DIMAP,
University of Warwick,
Coventry CV4 7AL, UK. Email: o.pikhurko@warwick.ac.uk}
\and
Shumin Sun\footnote{Mathematics Institute,
University of Warwick,
Coventry CV4 7AL, UK. Email: shumin.sun@warwick.ac.uk}
\and
Yan Wang\footnote{School of Mathematical Sciences, Shanghai Jiao Tong University, Shanghai 200240, China. Email: yan.w@sjtu.edu.cn}
\and
Jiasheng Zeng\footnote{Department of Mathematics, Hong Kong University of Science and Technology, Clear Water Bay, Kowloon, Hong Kong. Email: jzengbl@connect.ust.hk}
}
\maketitle
\begin{abstract}
The famous and actively studied problem of  Brown--Erd\H{o}s--S\'os from 1973 asks for $f^{(r)}(n;s,k)$, the maximum number of edges in an $r$-graph with $n$ vertices in which no $s$ vertices span $k$ or more edges.

In this paper, we concentrate on the case $r=3$ and $s=k+2$, with $k\ge2$ fixed and $n\to\infty$; then it is easy to show that the extremal function grows quadratically in $n$.
Delcourt and Postle proved that the limit $\bes{k}:=\lim_{n\to\infty} f^{(3)}(n;k+2,k)/n^2$ exists for every $k$. 
While Brown, Erd\H{o}s and S\'os observed that $\bes{2}=1/6$ already in the 1970s, the value of $\bes{k}$ for $3\le k\le 7$ was determined only recently (by various subgroups of Glock, Joos, Kim, K\"uhn, Lichev, Pikhurko, and Sun). Very recently, Chao, Huang and Liu determined $\bes{k}$ for every odd $k$.

Independently of the last result, we show that $\bes{9}=1/5$. Also, we prove that $\bes{8}\le {5053}/{26544}$, which is within $0.0029$ of the best known lower bound $\bes{8}\ge 3/16$. The new upper bounds are obtained by expressing some previous arguments as a linear program and then using a computer to generate and solve its instances. Our proof of the lower bound on $\bes{9}$ is based on a finite field construction combined with existing packing results.

\hide{
We study the quadratic Brown--Erd\H{o}s--S\'os problem for $3$-uniform hypergraphs with eight and nine edges. We prove
\[
\pi(3,8)\le \frac{5053}{26544}
\qquad\text{and}\qquad
\pi(3,9)=\frac15.
\]
The upper bounds are computer-assisted. Building on the cluster-weighting method developed in~\cite{GKLPS24}, we formulate the choice of pair weights as a linear program giving the stated upper bound for the eight-edge problem. Together with a matching lower bound, we also determine the exact value in the nine-edge case.
}
\end{abstract}

\section{Introduction}

Given a family $\C F$ of \emph{$r$-graphs} (that is, $r$-uniform set systems), the \emph{Tur\'an number} of $\C F$, denoted by $\ex(n;\C F)$, is defined as the maximum number of edges in an $n$-vertex $r$-graph containing no element of $\C F$ as a subgraph.  

Brown, Erd\H{o}s and S\'os~\cite{SosErdosBrown73,71BES} initiated the systematic study of the Tur\'an numbers of the family 
$\C F^{(r)}(s,k)$, consisting of all $r$-graphs with $k$ edges and at most $s$ vertices. We denote these Tur\'an numbers by 
$$f^{(r)}(n;s,k):=\ex(n;\C F^{(r)}(s,k)).$$
For fixed integers $r\geq3$, $k\geq2$ and $r<s\leq rk$, Brown, Erd\H{o}s and S\'os~\cite{71BES} established the general bounds 
 \begin{equation}\label{eq:GeneralBounds}
\Omega(n^{(rk-s)/(k-1)})=f^{(r)}(n;s,k)=O(n^{\lceil (rk-s)/(k-1)\rceil}).
 \end{equation}
Observe that if $s=rk-tk+t$ for an integer $t$, then the two exponents match, implying that $f^{(r)}(n;rk-tk+t,k)=\Theta (n^t)$. Of course, if $s$ is less than $rk-tk+t$ then the left-hand side of~\eqref{eq:GeneralBounds} grows polynomially faster than $n^t$. On the other hand, determining the set of $s>rk-tk+t$ with $f^{(r)}(n;s,k)=\Omega(n^t)$ is a very difficult problem even for $3$-graphs. Namely,  Brown, Erd\H{o}s and S\'os conjectured that
$f^{(3)}(n;k+3,k)=o(n^2)$ for every fixed $k\geq 3$. While Ruzsa and Szemer\'edi~\cite{RuzsaSzemeredi78} proved the case $k=3$, all other cases remain open, including the case $k=4$ when we get the notoriously difficult $(7,4)$-problem.
For general $k$, S\'ark\"ozy and Selkow~\cite{SarkozySelkow05} proved that $f^{(3)}(n;k+2+\lfloor\log_2 k\rfloor,k)=o(n^2)$. Conlon, Gishboliner, Levanzov and Shapira~\cite{ConlonGishbolinerLevanzovShapira23} subsequently improved the number of vertices to $k+O(\log k/\log\log k)$ for large $k$. 

So by the \emph{quadratic case} of the Brown--Erd\H{o}s--S\'os problem we mean the case  $s=rk-2k+2$ when we know by~\eqref{eq:GeneralBounds} that $f^{(r)}(n;s,k)=\Theta (n^2)$. A natural question is whether the following limit exists:
$$\pi(r,k):=\lim_{n\to \infty}n^{-2}f^{(r)}(n;rk-2k+2,k).$$
The existence of this limit for $r=3$ was conjectured by Brown, Erd\H{o}s and S\'os~\cite{71BES}. 
This conjecture was recently proved by Delcourt and Postle~\cite{DelcourtPostle24}. Building on their approach, Shangguan~\cite{23sh} proved the existence of the limit for every $r\geq4$ and $k\geq3$. 

The existence of the limit $\pi(r,k)$ leads to the problem of determining its value. The cases $2\leq k\leq7$ have been resolved for every $r\geq3$. For $k=2$, the packing theorem of R\"odl~\cite{Rodl85} gives $\pi(r,2)=1/(r^2-r)$ for every $r\geq3$. For $k=3$, Glock~\cite{19G} determined $\pi(3,3)=1/5$, while Shangguan and Tamo~\cite{ShangguanTamo20} proved that $\pi(r,3)=1/(r^2-r-1)$ for every $r\geq4$. The cases $k=4$ and $k\in\{5,6,7\}$ were settled by
Glock, Joos, Kim, K\"uhn, Lichev and
Pikhurko~\cite{GlockJoosKimKuhnLichevPikhurko24}, and by
Glock, Kim, Lichev, Pikhurko and Sun~\cite{GKLPS24},
respectively. For $k=8$, Pikhurko and Sun~\cite{ps25} proved that $\pi(r,8)=1/(r^2-r)$ for every $r\geq4$. They also established the lower bound $\pi(3,8)\geq3/16$ and conjectured equality. For general even integers $k\geq4$, Letzter and Sgueglia~\cite{25LS} proved that $\pi(r,k)=1/(r^2-r)$ whenever $r\geq 2+\sqrt{2}\,k^{3/2}$. Wang and Zeng~\cite{WangZeng26} subsequently improved this range to
$r\geq 2+\sqrt{3k/2-4}$. As this paper was in the final stages of its preparation, a preprint by Chao, Huang and Liu~\cite{ChaoHuangLiu26} appeared which, in particular, determined all remaining values of $\pi(r,k)$ except when $r=3$ and $k\ge 8$ is even.

In this paper, we focus on the quadratic case of the Brown--Erd\H{o}s--S\'os problem for 3-graphs (that is, the case $r=3$ and $s=k+2$). We write $\bes{k}:=\pi(3,k)$ for the corresponding limiting value. The results mentioned in the previous paragraph determined the values of $\bes{k}$ for all odd $k$ and for even $k\le 6$.  

Our first result is that $\bes{9}=1/5$. It is obtained independently of~\cite{ChaoHuangLiu26} and our proof of the lower bound uses a somewhat different method (a finite field construction combined with existing packing results) than that in~\cite{ChaoHuangLiu26}.

\begin{theorem}\label{thm:9-edge}
    $\bes{9}=\frac{1}{5}$.
\end{theorem}

The smallest $k$ for which the value of $\bes{k}$ is not known is $k=8$. The general upper bound of Delcourt and Postle~\cite{DelcourtPostle24} gives $\bes{8}\leq7/30$ while the newer results of Chao, Huang and Liu~\cite[Theorems 2.3 and 2.4]{ChaoHuangLiu26} imply that $\bes{8}\le 1/5$. Our second result improves this upper bound.

\begin{theorem}\label{thm:8-edge-ub}
    $\bes{8}\le \frac{5053}{26544}=0.19036...$\,.
\end{theorem}

Note that this bound exceeds the lower bound $3/16$ by at most $0.0029$, supporting the conjecture from~\cite{ps25} that $\bes{8}=3/16$.
The upper bounds in both theorems are obtained by generating and solving LPs (linear programs) using a computer. It is possible that the LP approach to $\bes{8}$ gives better bounds, perhaps even $3/16$, but we had to restrict our search to LPs that could be generated and solved with the available computational resources.

\hide{
\medskip 
\noindent \textbf{Added note.} As this paper was in the final stages of its preparation, a new preprint by Chao, Huang and Liu~\cite{ChaoHuangLiu26} appeared that, in particular, completely supersedes our Theorem~\ref{thm:9-edge}. Since our lower bound in Theorem~\ref{thm:9-edge} is proved by a method different from that in~\cite{ChaoHuangLiu26} and the results in the new preprint do not imply the upper bound in Theorem~\ref{thm:8-edge-ub}, we decided to keep both theorems.
}

\medskip
\noindent\textbf{Organisation.} The rest of this paper is organised as follows. We introduce some necessary definitions and notation in Section~\ref{se:prelim}. 
Section~\ref{se:lower} is devoted to the lower bound in Theorem~\ref{thm:9-edge}. The proofs of the upper bounds in Theorems~\ref{thm:9-edge} and~\ref{thm:8-edge-ub} can be found in Section~\ref{se:upper}. These rely on computer calculations; details on their implementation can be found in 
Appendix~\ref{se:appendix-computation}.
The SageMath code and certificates can be found in the ancillary folder on arXiv. 

\section{Preliminaries}\label{se:prelim}
Throughout the paper, we use the following notation and terminology. Let $\mathbb{N}$ denote the set of positive integers. For $n\in \mathbb{N}$, we denote by $[n]$ the set $\{1,\dots,n\}$. For a set $X$ and an integer $s \ge 1$, we let 
\[\binom{X}{s}:=\{Y\subseteq X: |Y|=s\}\] 
be the family of all $s$-element subsets of $X$. We will often write an unordered pair $\{x,y\}$ (resp.\ triple $\{x,y,z\}$) simply as $xy$ (resp.\ as $xyz$).  

When some definition or result involving $r$-graphs is equally easy to state for general $r$ as for $r=3$, we usually give  the general form.
When it is notationally convenient, we may identify an $r$-graph $H$ with its set of edges. If we specify only the edge set $E(H)$, then the vertex set is assumed to be the union of these edges, that is, $V(H):=\bigcup_{X\in E(H)} X$. Let $|H|$ be the number of edges of $H$ and $v(H)$ be the number of vertices of~$H$.
For $r$-graphs $F$ and $H$, their \emph{union} $F\cup H$ and \emph{difference} $F\setminus H$ have edge sets $E(F)\cup E(H)$ and $E(F)\setminus E(H)$ respectively. 

A \emph{diamond} consists of two edges whose symmetric difference has 2 elements.
For positive integers $s$ and $k$, an \emph{$(s,k)$-configuration} is an $r$-graph with $k$ edges and at most $s$ vertices, that is, an element of $\mathcal{F}^{(r)}(s,k)$. An $r$-graph is called \emph{$(s,k)$-free} if it contains no $(s,k)$-configuration. 
We adopt the following terminology:
\begin{itemize}
    \item If $s=rk-2k+2$, we simply refer to an $(s,k)$-configuration as a \emph{$k$-configuration}.
    \item If $s=rk-2k+1$, we say that it is a \emph{$k^-$-configuration}.
\end{itemize}

Let us define another $r$-graph family $\mathcal{G}^{(r)}_{k}$ which plays a central role in our proofs:
\begin{equation}\label{eq:cG}
\mathcal{G}^{(r)}_{k}:=\mathcal{F}^{(r)}(rk-2k+2,k) \cup\left(\bigcup_{\ell=2}^{k-1}\mathcal{F}^{(r)}(r\ell-2\ell+1,\ell)\right).
\end{equation}
Thus, $\mathcal{G}^{(r)}_{k}$ is the family of all $k$-configurations, as well as all analogous $r$-graphs of smaller sizes $\ell < k$ that are ``denser'' (that is, they are $\ell^-$-configurations). The family $\mathcal{G}^{(r)}_{k}$ will be considered for both the lower and the upper bounds.

We adopt the following definitions used in both~\cite{GKLPS24} and~\cite{ps25}. For an $r$-graph $G$, a pair $xy$ of distinct vertices (not necessarily in ${V(G)\choose 2}$) and $A\subseteq\mathbb N\cup\{0\}$, we say that $G$ \emph{$A$-claims} the pair $xy$ if, for every $i\in A$, there are $i$ distinct edges $e_1,\dots,e_i\in E(G)$ such that $|\{x,y\}\cup (\bigcup_{j=1}^i e_j)|\leq ri-2i+2$. In particular, if $xy\in {V(G)\choose 2}$, this is equivalent to the existence, for every positive $i\in A$, of an $i$-configuration $J\subseteq G$, possibly with isolated vertices, such that $\{x,y\}\subseteq V(J)$. Let $\p{A}{G}$ be the set of all pairs in ${V(G)\choose 2}$ that are $A$-claimed by~$G$.  If $A=\{i\}$ is a singleton, we simply write \emph{$i$-claims} instead of $\{i\}$-claims, and abbreviate  $\p{\{i\}}{G}$ to  $\p{i}{G}$.
For $i=1$, $\p{1}{G}$ is the usual \emph{$2$-shadow} of $G$ consisting of all pairs $uv$ of vertices such that there exists some edge $e\in E(G)$ with $u,v\in e$.
Let $\CI_{G}(xy)$ be the set of those $i\geq 0$ such that the pair $xy$ is $i$-claimed by $G$, that is,
 \begin{equation*}
 \textstyle
 \CI_{G}(xy):=\left\{ i\geq 0: \exists\mbox{ distinct $e_1,\dots,e_i\in G$ with }\big|\{x,y\}\cup (\bigcup_{j=1}^i e_j)\big|\leq ri-2i+2\right\}.
 \end{equation*}
\hide{More generally, for disjoint subsets $A,B\subseteq \I N$, we say that $G$ \emph{\OutIn{A}{B}-claims} a pair $xy$ if $A\cap \CI_{G}(xy)=\varnothing$ and $B\subseteq \CI_{G}(xy)$. For simplicity, we often omit curly brackets. For example, when $A=\{1\}$ and $B=\{i\}$ we just say \emph{\OutIn{1}{i}-claims}; also, we  let $\pp{1}{i}{G}:=\p{i}{G}\setminus \p{1}{G}$ denote the set of pairs in ${V(G)\choose 2}$ that are \OutIn{1}{i}-claimed by~$G$, and similarly let $\pp{12}{i}{G}:=\p{i}{G}\setminus (\p{1}{G}\cup \p{2}{G})$.
}

\section{Lower bounds}\label{se:lower}
For the lower bounds, we rely on the following result by Glock, Joos, Kim, K\"{u}hn, Lichev and Pikhurko~\cite{GlockJoosKimKuhnLichevPikhurko24} that allows us to build relatively dense $r$-graphs $G$ containing no $k$-configuration from a fixed $\mathcal{G}^{(r)}_{k}$-free $r$-graph $F$.

\begin{theorem}[\cite{GlockJoosKimKuhnLichevPikhurko24}*{Theorem 3.1}]\label{thm:highgirth}
Fix $k\geq 2$ and $r\geq 3$. Let $F$ be a non-empty $\cG{r}{k}$-free $r$-graph. 
Then,
$$\liminf_{n\rightarrow \infty}\frac{f^{(r)}(n;rk-2k+2,k)}{n^2}\geq \frac{|F|}{2\,|\p{\le \lfloor k/2\rfloor}{F}|},$$
where we define $P_{\le t}(F):=\{xy\in {V(F)\choose 2}: \CI_{F}(xy)\cap [t]\not=\varnothing\}$ to consist of all pairs $xy$ of vertices of $F$ such that $\CI_{F}(xy)$ contains some $i$ with $1\leq i\leq t$.
\end{theorem}

At a high level, the strategy to prove Theorem~\ref{thm:highgirth} is to find an almost perfect edge-packing of $J := \p{\le \lfloor k/2\rfloor}{F}$ in the $n$-vertex clique, and subsequently place a copy of $F$ onto each $J$. Since each copy of $F$ contains no $k$-configuration, one only needs to prevent $k$-configurations arising from at least two different copies of $J$. To achieve this, the packing has to be chosen carefully, using the general theory of conflict-free hypergraph matchings developed independently by Glock, Joos, Kim, K\"{u}hn and Lichev~\cite{24GJKKL} as well as Delcourt and Postle~\cite{DelcourtPostle26}.

We remark that Theorem~\ref{thm:highgirth} was utilised to verify the existence of the limit $\pi(r,k)$ in~\cite{DelcourtPostle24, 23sh}, and to establish lower bounds in~\cite{GlockJoosKimKuhnLichevPikhurko24,GKLPS24,ps25,ChaoHuangLiu26}. In this paper, we will also employ this theorem to derive the lower bound in Theorem~\ref{thm:9-edge}.

Before giving the details, we briefly mention some ideas of the proof. The basic block used in the proof is the following $3$-graph $Q$. Given $d$ pairwise disjoint pairs $L_1,\ldots,L_d$, which we call \textit{core pairs}, we let $x,y$ be two additional vertices called \emph{private}, and consider the 3-graph
\[
Q:=\{L_ix,L_i y:i\in[d]\},
\]
 where $L_i x$ denotes the triple $L_i\cup\{x\}$, and similarly for $L_i y$. Thus the $3$-graph $Q$ consists of $d$ diamonds sharing the same two vertices $x,y$. Although the pair $xy$ contributes $1$ to $|P_{\le 4}(Q)|$ (in fact, $xy$ is already in $P_{\le 2}(Q)$), it is in $d$ diamonds of $Q$ so its contribution per diamond is only $1/d$ (which will be negligible as we will let $d\to \infty$). Of course, the \emph{crossing pairs} of $Q$, that is, those connecting $L_i$ to $L_j$ for $1\le i<j\le d$, are also in $P_{\le 4}(Q)$. We would like to take different copies of $Q$ with each $Q$ having its own private pair so that each crossing pair appears on average in many blocks. For this, we construct a `sparse' graph $\Gamma$ and a collection $M$ of $d$-tuples of disjoint core pairs, and then place one copy of the block $Q$ (with two new private vertices) on each element of $M$ to obtain a $3$-graph $F_{d,q}$. The construction of $\Gamma$ and $M$ uses a finite field with $q$ elements.
 The proof then follows from two claims: $F_{d,q}$ is $\cG{3}{9}$-free, and $|\p{\le 4}{F_{d,q}}|\le(5d+1)N+|E(\Gamma)|$, where $N$ is the number of blocks. 
 We then apply Theorem~\ref{thm:highgirth} for fixed $d$ with $q\to\infty$. Finally, we let $d\to\infty$.

We now construct the sparse support graph and the collection $M$ of $d$-tuples of core pairs. Note that in the definition of a matching in Item~\ref{it:0} below, while we require that the pairs $L_i$ and $L_i'$ are not equal, we allow $L_i\cap L_i'$ to be non-empty
(and, in fact, there must be many such intersections).  

\begin{lemma}\label{lem:k9-support}
Fix an integer $d\ge 3$. For every prime power $q$, there exist pairwise disjoint sets $A_{i,a}$ for $(i,a)\in[d]\times \{0,1\}$, each of size $q^d$, a graph $\Gamma$ on $\bigcup_{i\in[d]}(A_{i,0}\cup A_{i,1})$, and a set $M\subseteq\C L_1\times\cdots\times\C L_d$, where
\[
\C L_i:= \left\{\{u,v\}:u\in A_{i,0},\ v\in A_{i,1}\right\},\quad \text{for $i \in [d]$},
\]
with the following properties.
\begin{enumerate}[label=(\roman*)]
	\item\label{it:0} The set $M$ is a \emph{matching}, that is, for all distinct $d$-tuples $(L_1,\dots,L_d),(L_1',\dots,L_d')\in M$ and every $i\in [d]$ it holds that $L_i\not=L_i'$.
    \item\label{it:1} $|M|=q^{2d}$.
    \item\label{it:2} If $(L_1,\ldots,L_d)\in M$, then $\{uv:u\in L_i,\ v\in L_j\}\subseteq E(\Gamma)$ for every $1\le i<j\le d$.
    \item\label{it:3} $|E(\Gamma)|\le 4\binom{d}{2}q^{2d-1}$.
\end{enumerate}
\end{lemma}

\begin{proof}
Let $V:=\FF_q^{2d}$ and fix a basis $e_1,\ldots,e_d,f_1,\ldots,f_d$ of $V$. For each $i\in[d]$, define two $d$-dimensional subspaces of $V$:
$$
\begin{aligned}
K_{i,0}&:=\langle e_i\rangle+\langle f_j:j\in[d]\setminus\{i\}\rangle,\\
K_{i,1}&:=\langle f_i\rangle+\langle e_j:j\in[d]\setminus\{i\}\rangle.
\end{aligned}
$$

For each $i\in[d]$ and $a\in\{0,1\}$, let $A_{i,a}$ be a copy of the quotient space $V/K_{i,a}$. We assume that all the sets $A_{i,a}$ are pairwise disjoint and write $\pi_{i,a}:V\to A_{i,a}$ for the quotient map. For every $i\in[d]$, we have $K_{i,0}\cap K_{i,1}=\{0\}$. Indeed, a vector in $K_{i,0}$ has possible $e$-coordinate only in the direction $e_i$ and possible $f$-coordinates only in the directions $f_j$ with $j\ne i$, while a vector in $K_{i,1}$ has possible $f$-coordinate only in the direction $f_i$ and possible $e$-coordinates only in the directions $e_j$ with $j\ne i$. Hence no non-zero vector can occur in both subspaces.

For each $t\in V$ and $i\in[d]$, we define the pair 
\[
L_i(t):=\left\{\pi_{i,0}(t),\pi_{i,1}(t)\right\}.
\]
Define 
 \[
 M:=\{(L_1(t),\ldots,L_d(t)):t\in V\}\subseteq \C L_1\times\cdots\times\C L_d.
 \]
 Since $K_{i,0}\cap K_{i,1}=\{0\}$, the linear map $t\mapsto (\pi_{i,0}(t),\pi_{i,1}(t))$ is injective. Since we assumed that $A_{i,0}\cap A_{i,1}=\emptyset$, the map $t\to \{\pi_{i,0}(t),\pi_{i,1}(t)\}$ is also injective. 
 Thus, for each $i$, the pairs $L_i(t)$, $t\in V$, are all distinct and  $M$ is a matching. Clearly, $|M|=|V|=q^{2d}$. This shows Items~\ref{it:0} and~\ref{it:1}.

Next, we define the support graph. For every $1\le i<j\le d$, every $a,b\in\{0,1\}$, and every $t\in V$, put the pair $\{\pi_{i,a}(t),\pi_{j,b}(t)\}$ into $E(\Gamma)$. By construction, if $(L_1(t),\ldots,L_d(t))\in M$, then $\{uv:u\in L_i(t),\ v\in L_j(t)\}\subseteq E(\Gamma)$ for every $1\le i<j\le d$. This proves Item~\ref{it:2}.

We now bound $|E(\Gamma)|$. Fix $i<j$ and $a,b\in\{0,1\}$. The number of distinct pairs of the form $\{\pi_{i,a}(t),\pi_{j,b}(t)\}$, where $t\in V$, is the size of the image of the linear map $t\mapsto(\pi_{i,a}(t),\pi_{j,b}(t))$. Its kernel is $K_{i,a}\cap K_{j,b}$, and this intersection contains a non-zero vector whenever $i\ne j$. Indeed, if $a=b=0$, then, since $d\ge3$, there exists $h\in[d]\setminus\{i,j\}$, and $f_h\in K_{i,0}\cap K_{j,0}$. Similarly, if $a=b=1$, then $e_h\in K_{i,1}\cap K_{j,1}$. If $(a,b)=(0,1)$, then $e_i\in K_{i,0}\cap K_{j,1}$, while if $(a,b)=(1,0)$, then $f_i\in K_{i,1}\cap K_{j,0}$.

Thus $\dim(K_{i,a}\cap K_{j,b})\ge1$, so the image of the linear map $t\mapsto(\pi_{i,a}(t),\pi_{j,b}(t))$ has size at most $q^{2d-1}$. Summing over the $4\binom{d}{2}$ choices of $(i,j,a,b)$ gives $|E(\Gamma)|\le4\binom{d}{2}q^{2d-1}$.
\end{proof}

Fix an integer $d\ge3$ and a prime power $q$, and let $\Gamma$ and $M$ be the graph and matching obtained by Lemma~\ref{lem:k9-support}, with $V:=\FF_q^{2d}$. Write $N:=|M|=q^{2d}$. Then $|E(\Gamma)|\le4\binom{d}{2}q^{2d-1}=o(N)$ for fixed $d$ as $q\to\infty$.

For every matching element $(L_1(t),\ldots,L_d(t))\in M$, where $t\in V$, introduce two new private vertices $x_t,y_t$. These private vertices are distinct for different $t$ and are disjoint from all vertices of the support graph. Define the \emph{block}
$$
Q_t:=\{L_i(t)x_t,L_i(t)y_t:i\in[d]\},\quad\text{for $t\in V$}.
$$
 Recall that e.g.\ $L_i(t)x_t$ is a shorthand for the triple $L_i(t)\cup\{x_t\}$. Finally, let
$$
F_{d,q}:=\bigcup_{t\in V}Q_t.
$$
Then $|F_{d,q}|=2dN$.

\begin{lemma}\label{lem:k9-free}
The $3$-graph $F_{d,q}$ is $\cG{3}{9}$-free.
\end{lemma}

\begin{proof}
Since two distinct triples span at least four vertices, it suffices to prove that every set of $h$ edges of $F_{d,q}$ spans at least $h+2$ vertices for $3\le h\le8$, and every set of nine edges spans at least $12$ vertices.

Let $S\subseteq E(F_{d,q})$ be a set of $h$ edges, where $3\le h\le9$. For $t\in V$ and $i\in[d]$, call the index $(t,i)$ \textit{active} if $S$ contains at least one of the two edges $L_i(t)x_t$ and $L_i(t)y_t$, and call it \textit{double active} if $S$ contains both edges. Let $c$ be the number of active indices and $a$ the number of double active indices. Then $h=c+a$, since $a$ of the $c$ active indices
contribute two edges to $S$ and the remaining $c-a$ contribute
one edge each.

For each $i\in[d]$, set $T_i:=\{t\in V:(t,i)\text{ is active}\}$ and $m_i:=|T_i|$. Let $I:=\{i\in[d]:m_i>0\}$, $s:=|I|$, and $m_{\max}:=\max_{i\in I}m_i$. Thus $\sum_{i\in I}m_i=c$.

We first lower-bound the number of core vertices used by $S$. For a positive integer $m\le9$, define $g(m):=\min\{u+v:u,v\in\mathbb N,\ uv\ge m\}$. We compute $g(m)$ for $m \in [9]$:
$$
\begin{array}{c|ccccccccc}
m&1&2&3&4&5&6&7&8&9\\
\hline
g(m)&2&3&4&4&5&5&6&6&6
\end{array}
$$
Fix $i\in I$. Since the map $t\mapsto L_i(t)$ is injective, the $m_i$ pairs $L_i(t)$ with $t\in T_i$ are distinct pairs between $A_{i,0}$ and $A_{i,1}$. If their endpoints occupy $u_i$ vertices in $A_{i,0}$ and $v_i$ vertices in $A_{i,1}$, then $m_i\le u_iv_i$. Hence these $m_i$ pairs use at least $g(m_i)$ core vertices. Since the sets $A_{i,a}$ are pairwise disjoint for different $(i,a)$, the total number of core vertices used by $S$ is at least $\sum_{i\in I}g(m_i)$.

We next lower-bound the number of private vertices used by $S$. Call a block $Q_t$ \emph{active} if it contains at least one edge from $S$. Let $B:=\{t\in V:S\cap Q_t\ne\emptyset\}$ be the set of active blocks. Since some coordinate $i$ satisfies $m_i=m_{\max}$, there are at least $m_{\max}$ active blocks, so $|B|\ge m_{\max}$. Let $B_2$ be the set of blocks containing at least one double active index. Every block can contain double active indices only in its active coordinates, at most $s$ of them; so $|B_2|\ge\lceil a/s\rceil$. Each active block contributes at least one of its two private vertices $x_t,y_t$, and every block in $B_2$ contributes both private vertices. Therefore $v(S)\ge\sum_{i\in I}g(m_i)+m_{\max}+\lceil a/s\rceil$.

It remains to minimize this lower bound for $h\le9$. For fixed $c$ and $a$, define
$$
\nu(c,a):=
\min_{\substack{s\ge1,\ m_1,\ldots,m_s\ge1\\
                m_1+\cdots+m_s=c}}
\left(
   \sum_{\rho=1}^s g(m_\rho)
   +\max_{\rho\in[s]}m_\rho
   +\left\lceil\frac as\right\rceil
\right).
$$
This is a relaxation of the actual situation, since we do not impose $s\le d$, and therefore it still gives a valid lower bound. Since $a=h-c$ and $a\le c$, the possible values of $c$ are $\lceil h/2\rceil,\lceil h/2\rceil+1,\ldots,h$. Substituting the displayed values of $g$, we obtain the following table (the values of the second column are listed in this order of $c$):
$$
\begin{array}{c|c|c}
h&
\{\nu(c,h-c):\lceil h/2\rceil\le c\le h\}&
\min_{\lceil h/2\rceil\le c\le h}\nu(c,h-c)\\
\hline
3&6,\ 7&6\\
4&6,\ 8,\ 8&6\\
5&8,\ 9,\ 10&8\\
6&8,\ 9,\ 11,\ 11&8\\
7&10,\ 11,\ 12,\ 12&10\\
8&10,\ 11,\ 12,\ 13,\ 12&10\\
9&12,\ 12,\ 13,\ 13,\ 14&12
\end{array}
$$

Consequently, if $3\le h\le8$, then $v(S)\ge h+2>h+1$, so $S$ is not an $h^-$-configuration. If $h=9$, then $v(S)\ge12>11$, so $S$ is not a $9$-configuration. Thus $F_{d,q}$ contains no $\ell^-$-configuration for $3\le\ell\le8$ and no $9$-configuration. Hence $F_{d,q}$ is $\cG{3}{9}$-free.
\end{proof}

\begin{lemma}\label{lem:k9-claimed-pairs}
The claimed-pair set of $F_{d,q}$ satisfies $|\p{\le 4}{F_{d,q}}|\le(5d+1)N+|E(\Gamma)|$.
\end{lemma}

\begin{proof}
Set $R:=\p{1}{F_{d,q}}\cup\{x_ty_t:t\in V\}\cup E(\Gamma)$. We first show that $\p{\le 4}{F_{d,q}}\subseteq R$.
Now let $uv\in\p{\le 4}{F_{d,q}}$. Then for some $p\in\{1,2,3,4\}$ there are $p$ distinct edges $e_1,\ldots,e_p\in E(F_{d,q})$ such that 
\begin{equation}\label{eq:cl-pair}
\left|\{u,v\}\cup\bigcup_{\alpha=1}^p e_\alpha\right|\le p+2.
\end{equation}
We distinguish the four possible values of $p$.

If $p=1$, then the right-hand side of~(\ref{eq:cl-pair}) is equal to $3$, so both $u$ and $v$ lie in $e_1$. Thus $uv\in\p{1}{F_{d,q}}\subseteq R$.

Suppose $p=2$. Note that two distinct triples have union of size at least $4$.
Hence $|e_1\cup e_2|=4$ and $\{u,v\}\subseteq e_1\cup e_2$. Thus $e_1$ and $e_2$ share exactly two vertices. In $F_{d,q}$ this can happen only when, for some $t\in V$ and $i\in[d]$, $\{e_1,e_2\}=\{L_i(t)x_t,L_i(t)y_t\}$. Indeed, two edges can share two core vertices only if they contain the same core pair $L_i(t)$, and the injectivity of $t\mapsto L_i(t)$ then forces the same $t$ and the same $i$; the distinctness of the two edges implies that the private vertices must be $x_t$ and $y_t$. Also, two distinct edges cannot share one private vertex and one core vertex: the private vertex determines the block $t$, and inside one block the different core pairs $L_i(t)$ are disjoint. Therefore $e_1\cup e_2=L_i(t)\cup\{x_t,y_t\}$. Every pair inside this $4$-set is in the usual $2$-shadow or is the private pair $\{x_t,y_t\}$. Hence $uv\in R$.

Suppose $p=3$. The proof of Lemma~\ref{lem:k9-free} shows that any three edges of $F_{d,q}$ span at least $6$ vertices, whereas the right-hand side of~(\ref{eq:cl-pair}) is $5$. So this case cannot occur.

It remains to consider $p=4$. Lemma~\ref{lem:k9-free} shows that any four edges of $F_{d,q}$ span at least $6$ vertices. Since the right-hand side of~(\ref{eq:cl-pair}) is now $6$, we have $\left|\cup_{\alpha=1}^4e_\alpha\right|=6$ and $\{u,v\}\subseteq\bigcup_{\alpha=1}^4e_\alpha$. We claim that the four edges must be two complete diamonds in one block.

Use the notation from the proof of Lemma~\ref{lem:k9-free} for $S:=\{e_1,e_2,e_3,e_4\}$. Thus $c$ is the number of active indices and $a$ is the number of double active indices. Since $4=c+a$ and $a\le c$, the only possibilities are $(c,a)=(4,0),(3,1),(2,2)$. If $(c,a)=(4,0)$ or $(3,1)$, then the table in the proof of Lemma~\ref{lem:k9-free} shows that we have at least $8$ vertices. Hence equality with $6$ vertices is possible only when $(c,a)=(2,2)$.

In this case both active indices are double. If the two active indices have the same coordinate, then the lower bound is at least $g(2)+2+2=7$, which is impossible. Therefore they have two different coordinates. Since the total number of vertices is exactly $6$, equality must also hold in the private-vertex lower bound, so there is only one active block. Thus there exist $t\in V$ and distinct $i,j\in[d]$ such that
$$
\{e_1,e_2,e_3,e_4\}=\{L_i(t)x_t,L_i(t)y_t,L_j(t)x_t,L_j(t)y_t\}.
$$
Consequently, $\cup_{\alpha=1}^4e_\alpha=L_i(t)\cup L_j(t)\cup\{x_t,y_t\}$. All pairs inside this $6$-set are of one of the following types: usual $2$-shadow pairs, the private pair $x_ty_t$, or cross-pairs between $L_i(t)$ and $L_j(t)$. The first type lies in $\p{1}{F_{d,q}}$, the second lies in $\{x_ty_t:t\in V\}$, and the cross-pairs lie in $E(\Gamma)$ by Lemma~\ref{lem:k9-support}. Hence $uv\in R$.

We have shown that $\p{\le 4}{F_{d,q}}\subseteq R$.
Note that $|\p{1}{F_{d,q}}|\le5dN$. Indeed, for fixed $t\in V$ and $i\in[d]$, the two edges $L_i(t)x_t$ and $L_i(t)y_t$ contribute to the usual $2$-shadow only the following five pairs: the core pair $L_i(t)$, the two pairs joining $x_t$ to the vertices of $L_i(t)$, and the two pairs joining $y_t$ to the vertices of $L_i(t)$. Summing over all $t\in V$ and $i\in[d]$ gives at most $5dN$ pairs.
Therefore
$$
\begin{aligned}
|\p{\le 4}{F_{d,q}}|
&\le|\p{1}{F_{d,q}}|+\bigl|\{x_ty_t:t\in V\}\bigr|+|E(\Gamma)| \le(5d+1)N+|E(\Gamma)|.
\end{aligned}
$$
This finishes the proof of the lemma.\end{proof}

We are now ready to prove the lower bound in Theorem~\ref{thm:9-edge}.

\begin{proof}[Proof of the lower bound in Theorem~\ref{thm:9-edge}]
Fix $d\ge3$ and let $q$ be an arbitrary prime power. By Lemmas~\ref{lem:k9-free} and~\ref{lem:k9-claimed-pairs}, the $3$-graph $F_{d,q}$ is $\cG{3}{9}$-free and satisfies $|\p{\le 4}{F_{d,q}}|\le(5d+1)N+|E(\Gamma)|$. Since $N=q^{2d}$ and $|E(\Gamma)|\le4\binom{d}{2}q^{2d-1}$, Theorem~\ref{thm:highgirth} implies that
$$
\bes{9}
\ge
\frac{2dq^{2d}}{2\bigl((5d+1)q^{2d}+4\binom{d}{2}q^{2d-1}\bigr)}
=
\frac{dq}{(5d+1)q+4\binom{d}{2}}.
$$
Letting first $q\to\infty$ and then $d\to\infty$ gives $\bes{9}\ge1/5$.
\end{proof}

\section{Upper bounds}\label{se:upper}

For the upper bound, we will use the following lemma, proved by Delcourt and Postle~\cite{DelcourtPostle24}*{Theorem 2.1} for $r=3$ and by Shangguan~\cite{23sh}*{Lemma 5} for $r\ge 4$, which enables us to exclude smaller ``denser'' structures when considering the limit.

\begin{lemma}[\cite{23sh}*{Lemma 5}]\label{lem:ub}
    For all fixed $r\ge 3$ and $k\ge 3$,
    $$\limsup_{n\to \infty}\frac{f^{(r)}(n;rk-2k+2,k)}{n^2}\le \limsup_{n\to \infty}\frac{\ex(n;\cG{r}{k})}{n^2}.$$
\end{lemma}

By Lemma~\ref{lem:ub}, we are able to restrict our attention to $\cG{r}{k}$-free $r$-graphs when proving the upper bound on $\pi(r,k)$. To upper bound $|G|$ for any $n$-vertex $\cG{r}{k}$-free $r$-graph $G$, we employ a \emph{weighting method}, which is also used in the previous work~\cite{GlockJoosKimKuhnLichevPikhurko24,GKLPS24,ps25}.

In brief, the method proceeds as follows.
We begin with a trivial partition of $E(G)$ into single edges and iteratively merge parts using specific rules (which depend on $k$ and $r$). Then, we assign a set of weights that each final part (a subgraph $F\subseteq G$) attributes to some of the pairs in $\binom{V(F)}{2}$, so that every vertex pair receives total weight at most $1$. Consequently, the total weight assigned by the parts is at most $\binom{n}{2}$, and this enables us to derive an upper bound on $|G|$ by upper bounding the ratio of the number of edges to the sum of weights for each possible final part.

Once the merging rules are fixed, finding an optimal weight assignment amounts to solving a (possibly infinite) linear programming (LP) problem. Our computer code tries to generate and solve this LP. 
One difficulty of this method lies in designing the merging rules: they must be fine enough so that the LP gives a good bound and yet coarse enough so that the LP is finite and of manageable size.


\subsection{Further definitions}

Recall that for an $r$-graph $F$ and a pair $uv$, the set $\CI_F(uv)$ consists of all integers $j\ge 0$ such that $F$ has $j$ edges that together with $uv$ include at most $rj-2j+2$ vertices. Note that, by definition, $0$ always belongs to $\CI_F(uv)$, which is notationally convenient in the subsequent context.

Let $G$ be an arbitrary $r$-graph. When dealing with a partition $\C P$ of $E(G)$, we will view each element $F\in\C P$ as an $r$-graph whose vertex set is the union of the edges in~$F$. Let $A,B\subseteq \mathbb N$ be any (not necessarily disjoint) sets of positive integers. For two subgraphs $F,H\subseteq G$, if they are edge-disjoint and there is a pair $uv$ such that $A\subseteq \CI_F(uv)$ and $B\subseteq \CI_H(uv)$, then we say that $F$ and $H$ are \emph{$\OMerge{A}{B}$-mergeable (via $uv$)}. Note that this relation is not symmetric in $F$ and $H$. When the ordering of the two $r$-graphs does not matter, we use the shorthand \emph{$\Merge{A}{B}$-mergeable} to mean $\OMerge{A}{B}$-mergeable or $\OMerge{B}{A}$-mergeable. Similar to the claim notation, for singletons we may omit the brackets, for example $\Merge{i}{j}$ means $\Merge{\{i\}}{\{j\}}$.

For sets $A,B\subseteq\mathbb N\cup\{0\}$, write
$
A+B:=\{a+b:a\in A,\ b\in B\}
$
for their Minkowski sum. If $F$ and $H$ are $\Merge{A}{B}$-mergeable via $uv$, then
\begin{equation}\label{eq:AB}
A\cup B \subseteq \CI_{F\cup H}(uv)\quad \text{and}\quad A+B\subseteq \CI_{F\cup H}(uv).
\end{equation}

For a finite collection of merging rules
$
\C R=\{\Merge{A_1}{B_1},\dots,\Merge{A_s}{B_s}\},
$
we say that two edge-disjoint subgraphs are \emph{$\C R$-mergeable} if they are mergeable under at least one rule in $\C R$. Starting with the partition of $E(G)$ into single edges, iteratively and as long as possible, replace two distinct $\C R$-mergeable parts by their union. The final partition is denoted by $\C M'^G_{\C R}$ and we call its elements \emph{clusters}.

\newcommand{\AllClusters}[2]{\mathfrak F_{\C R}(#1,#2)}
\newcommand{\ExtIntClusters}[2]{\mathfrak F_{\C R}'(#1,#2)}

It will be convenient to consider not only the clusters appearing from a particular graph, but the full closure of the merging rules over all starting graphs.
Namely, we define the family $\AllClusters rk$ of isomorphism types of $r$-graphs as follows. Initially, it consists of a single edge. As long as this increases the family,  apply the following \emph{abstract $\C R$-merging}: take any $F,H\in \AllClusters rk$ (possibly $F=H$), put their copies $F'$ and $H'$ on a common vertex set so that they are edge-disjoint (but are allowed to share vertices), and add the $r$-graph $F'\cup H'$ to $\AllClusters{r}{k}$ if it is $\cG{r}{k}$-free, is an $\C R$-merging of $F'$ and $H'$, and is not in $\AllClusters{r}{k}$. (For brevity, when talking about membership in $\AllClusters{r}{k}$, we will always mean up to isomorphism.) Thus $\AllClusters{r}{k}$ is the minimal family that contains the single edge and is closed under abstract $\C R$-mergings. Clearly, for every $\cG{r}{k}$-free $r$-graph $G$, $\C M'^G_{\C R}$ is a sub-family of $\AllClusters{r}{k}$.

\subsection{The linear programming problem}
In this section, we formalise the weighting method and translate the problem of upper bounding $\pi(r,k)$ into a linear programming problem.

For an $r$-graph $F$ and a pair $xy$, define the \emph{truncated claim set}
\[
\CI_F^{<k}(xy):=\CI_F(xy)\cap\{0,1,\dots,k-1\}.
\]
Only these claims are relevant when deciding whether a collection of claim sets can sum to $k$.

Let $\C R$ be a collection of merging rules $\Merge{A}{B}$, that we can choose depending on $r$ and $k$. In our proofs, $\C R$ will contain only merging rules $\Merge{A}{B}$ with  $A,B\subseteq[k-1]$; in particular, $\C R$ is finite. We say that a sequence  $S_1,S_2,\dots,S_m$ (of not necessarily distinct sets) possesses the \emph{$\C R$-property} if they satisfy the following:
\begin{enumerate}[label=(\roman*)]
    \item $0\in S_i\subseteq\{0,1,\dots,k-1\}$ for every $i\in[m]$;
    \item for every rule $\Merge{A}{B}\in\C R$ and every ordered pair of distinct indices $i,j\in[m]$, we have $A\not\subseteq S_i$ or $B\not\subseteq S_j$ (or both); \label{it:R-property}
    \item $k\notin S_1+\cdots+S_m$. \label{it:sum-property}
\end{enumerate}
 This definition is motivated by the fact (that is checked inside the proof of Theorem~\ref{thm:lp-ub}) that if we fix a pair $xy$ in $m$ final clusters $F_1,\dots,F_m$ then the corresponding truncated claim sets $S_i:=\CI_{F_i}^{<k}(xy)$ for $i\in [m]$ (taken in any order but keeping their multiplicities if some are repeated) have the $\C R$-property.

We consider the following linear program where the unknown variables are $t$ and $w_S$, where the latter are		 called \emph{weights} and are indexed by the $2^{k-1}$ subsets $S\subseteq\{0,1,\dots,k-1\}$ with $0\in S$: maximise $t$ subject to
\begin{align}
& \sum_{xy\in\binom{V(F)}2}w_{\CI_F^{<k}(xy)}\ge t|F|,
    &\ & \text{for every }
    F\in\AllClusters{r}{k},\label{lp:F}\\
& \sum_{i=1}^{m}w_{S_i}\le1,
    && \text{for every sequence }
    S_1,\dots,S_m\text{ with the $\C R$-property},\label{lp:S1...Sl}\\
& w_S\in[0,1],
    && \text{for every set $S$ with }
    0\in S\subseteq\{0,1,\dots,k-1\}.\label{lp:S}
\end{align}

\begin{remark}\label{rem:finite-lp} 
Note that there may be infinitely many sequences with the $\C R$-property and thus infinitely many constraints in~\eqref{lp:S1...Sl}. However, we set $w_{S}:=0$ for every $S$ such that the constant sequence $S^{(k)}:=(S,\dots,S)$ of length $k$ has the $\C R$-property
and then remove any duplicate inequalities in~\eqref{lp:S1...Sl}. Then every remaining sequence $S_1,\dots,S_m$ contains every subset $S$ at most $k-1$ times and the number of inequivalent constraints in~\eqref{lp:S1...Sl} is finite. Also, it is easy to see that this step does not change the value of the LP: if $S^{(k)}$ has the $\C R$-property then so does every $S^{(m)}$ and thus $w_S=0$ in any feasible solution to the original LP.\end{remark}
	
The number of constraints in~\eqref{lp:F} may be infinite and the authors do not know any general reduction that replaces $\AllClusters{r}{k}$ with a finite family without affecting the value of the LP. When we ran computer experiments, if the current family became too large then we had to abandon the used set of rules $\C R$.

The next result shows that any feasible solution of the (possibly infinite) LP gives an upper bound on $\pi(r,k)$.

\begin{theorem}\label{thm:lp-ub}
Let $r\ge3$ and $k\ge3$, and let $\C R$ be a collection of merging rules $\Merge{A}{B}$ with $A, B \subseteq [k-1]$. If $t>0$ is a feasible value of the linear program with constraints \eqref{lp:F}--\eqref{lp:S}, then 
\[
\pi(r,k)\le\frac1{2t}.
\]
\end{theorem}

\begin{proof}
Suppose $G$ is an arbitrary $\cG{r}{k}$-free $n$-vertex $r$-graph, and let
$\C{M}'^G_{\C R}$
be its final cluster partition. Recall that $\C{M}'^G_{\C R} \subseteq \AllClusters{r}{k}$, and hence, up to isomorphism, each cluster $F \in \C{M}'^G_{\C R}$ appears in the linear program.

For every final cluster $F\in \C{M}'^G_{\C R}$ and every pair $xy\in\binom{V(F)}2$, define the \emph{weight} of $xy$ contributed by $F$ to be $w_{\CI_F^{<k}(xy)}$. We then define
\[
w(F):=\sum_{xy\in\binom{V(F)}2}w_{\CI_F^{<k}(xy)},
\qquad F\in \C{M}'^G_{\C R},
\]
and
\[
w(xy):=\sum_{\substack{F\in \C{M}'^G_{\C R}\\xy\in\binom{V(F)}2}}w_{\CI_F^{<k}(xy)},
\qquad xy\in\binom{V(G)}2.
\]

We claim that
\begin{align}
    w(xy)&\le1 &&\text{for every }xy\in\binom{V(G)}2, \label{eq:one}\\
    w(F)&\ge t|F| &&\text{for every }F\in \C{M}'^G_{\C R}. \label{eq:weight}
\end{align}
Indeed,~\eqref{eq:weight} follows directly from~\eqref{lp:F}.

To prove~\eqref{eq:one}, fix a pair $xy\in\binom{V(G)}2$ and enumerate all final clusters containing $xy$ as $F_1,\dots,F_m$. If $m=0$, then $w(xy)=0$, so assume $m\ge1$. Set
\[
S_i:=\CI_{F_i}^{<k}(xy),\qquad i\in[m].
\]
We show that $S_1,\dots,S_m$ has the $\C R$-property. Clearly $0\in S_i$ for every $i$, so condition (i) is true. Moreover, no two distinct final clusters are $\C R$-mergeable, so condition~\ref{it:R-property} follows from the maximality of the merging procedure.

It remains to verify condition~\ref{it:sum-property}. Suppose for a contradiction that $k\in S_1+\cdots+S_m$, and choose $s_i\in S_i$ such that $\sum_{i=1}^m s_i=k$. For every $i$ with $s_i>0$, the cluster $F_i$ contains $s_i$ distinct edges which, together with $xy$, use at most 
$(r-2)s_i+2$ vertices. Since the clusters are edge-disjoint and all these witnesses contain the common pair $xy$, their union consists of $k$ distinct edges on at most
\[
2+\sum_{i=1}^m(r-2)s_i=(r-2)k+2=rk-2k+2
\]
vertices. This is a $k$-configuration in $G$, a contradiction. Thus $S_1,\dots,S_m$ has the $\C R$-property and the constraints in~\eqref{lp:S1...Sl} give
$
w(xy)=\sum_{i=1}^m w_{S_i}\le1.
$
This proves~\eqref{eq:one}.

Consequently,
\[
|G|=\sum_{F\in \C{M}'^G_{\C R}}|F|
\le t^{-1}\sum_{F\in \C{M}'^G_{\C R}}w(F)
=t^{-1}\sum_{xy\in\binom{V(G)}2}w(xy)
\le t^{-1}\binom n2.
\]
Since $G$ was arbitrary, Lemma~\ref{lem:ub} and $\binom n2 = (1+o(1)) n^2/2$ imply $\pi(r,k)\le1/(2t)$.
\end{proof}

\subsection{Proofs of the claimed upper bounds}

In brief, we derived the upper bounds of Theorems~\ref{thm:8-edge-ub} and~\ref{thm:9-edge} from Theorem~\ref{thm:lp-ub} by writing SageMath code that, for a given set of merging rules $\C R$, tries to generate the LP (reduced as in Remark~\ref{rem:finite-lp}) and solve it. When the size of $\AllClusters{r}{k}$ became too large for computer calculations, we would change $\C R$ to allow fewer mergings and try again. On the other hand, when we wanted to improve the current bound, we would enlarge~$\C R$. These (non-exhaustive) experiments with possible merging rules led to the presented bounds.

\begin{proof}[Proof of Theorem~\ref{thm:8-edge-ub}]
We apply Theorem~\ref{thm:lp-ub} with
\[
r=3,\qquad k=8,\qquad
\C R_8=\{\Merge{1}{1},\Merge{1}{2},\Merge{2}{2}\}.
\]
The procedure described in Appendix~\ref{se:appendix-computation} generates a finite family
$
\ExtIntClusters{3}8 \supseteq \AllClusters{3}{8}.
$
We impose the first family of LP constraints on all members of $\ExtIntClusters{3}8$; this can only restrict the feasible set of the LP in Theorem~\ref{thm:lp-ub}. Solving the resulting finite linear program using exact rational arithmetic gives a feasible solution with
$
t=\frac{13272}{5053}.
$
Therefore, Theorem~\ref{thm:lp-ub} gives
\[
\bes{8}\le\frac1{2t}=\frac{5053}{26544}.
\]
The computational details and the exact certificate are described in Appendix~\ref{se:appendix-computation}.
\end{proof}

\begin{proof}[Proof of the upper bound in Theorem~\ref{thm:9-edge}]
We apply Theorem~\ref{thm:lp-ub} with
\[
r=3,\qquad k=9,\qquad
\C R_9=\{\Merge{1}{1},\Merge{1}{2}\}.
\]
The procedure described in Appendix~\ref{se:appendix-computation} generates a finite family
$
\ExtIntClusters{3}9 \supseteq \AllClusters{3}{9}.
$
The resulting finite linear program has a feasible solution with
$
t=\frac52.
$
Hence Theorem~\ref{thm:lp-ub} gives
\[
\bes{9}\le\frac1{2t}=\frac15.
\]
\end{proof}

\section{Concluding remarks}

Whether the universal closure is finite depends on the choice of merging rules. For example, adding the rule $\Merge{1}{3}$ in the (still open) case when $k=8$  produces arbitrarily large clusters. Let $H$ be the $3$-graph which has vertex set $[5]$ with edges $123,124,345$. The pair $15$ is $3$-claimed by $H$, while $45$ is $1$-claimed. By taking a chain of copies of $H$ and identifying the pair $15$ in one copy with the pair $45$ in the next, one obtains arbitrarily large $\cG{3}{8}$-free clusters. Thus a finite full-closure enumeration is not available for this stronger rule set.

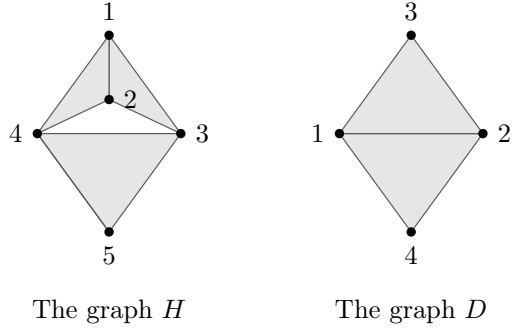
\begin{figure}[htbp]
\centering
\begin{tikzpicture}[
    font=\small,
    vertex/.style={circle,fill=black,inner sep=1.25pt,outer sep=0pt},
    hedge/.style={draw=black!65,fill=black!10,line width=0.45pt}
]
    \coordinate (h1) at (0,1.3);
    \coordinate (h2) at (0,0.45);
    \coordinate (h3) at (0.95,0);
    \coordinate (h4) at (-0.95,0);
    \coordinate (h5) at (0,-1.3);
    \path[hedge] (h1)--(h2)--(h3)--cycle;
    \path[hedge] (h1)--(h2)--(h4)--cycle;
    \path[hedge] (h3)--(h4)--(h5)--cycle;
    \draw[hedge] (h4)--(h5);
    \node[vertex,label=above:{$1$}] at (h1) {};
    \node[vertex,label=right:{$2$}] at (h2) {};
    \node[vertex,label=right:{$3$}] at (h3) {};
    \node[vertex,label=left:{$4$}] at (h4) {};
    \node[vertex,label=below:{$5$}] at (h5) {};
    \node at (0,-2.35) {The graph $H$};

    \coordinate (d1) at (3.05,0);
    \coordinate (d2) at (4.95,0);
    \coordinate (d3) at (4,1.3);
    \coordinate (d4) at (4,-1.3);
    \path[hedge] (d1)--(d2)--(d3)--cycle;
    \path[hedge] (d1)--(d2)--(d4)--cycle;
    \draw[hedge] (d1)--(d2);
    \node[vertex,label=left:{$1$}] at (d1) {};
    \node[vertex,label=right:{$2$}] at (d2) {};
    \node[vertex,label=above:{$3$}] at (d3) {};
    \node[vertex,label=below:{$4$}] at (d4) {};
    \node at (4,-2.35) {The graph $D$};
\end{tikzpicture}
\end{figure}

We also experimented with adding $\Merge{2}{3}$ to $\C R_8$. The resulting enumeration grows rapidly. We believe that the number of non-isomorphic clusters is finite, but from computational experiments it seems astronomical. The partial enumeration contained clusters with more than 100 vertices. On the portion generated in these experiments, the value $t=8/3$ was feasible. The value $8/3$ is naturally associated with packings of the above $H$-configuration where $3$-claimed pairs are shared by many copies.

For $k=9$, let $D$ be the diamond graph, which is the two-edge $3$-graph on vertex set $[4]$ with edges $123$ and $124$. The pair $34$ is $2$-claimed by $D$. Chaining copies of $D$ along these $2$-claimed pairs gives arbitrarily large $\cG{3}{9}$-free clusters, hence adding the $\Merge{2}{2}$ rule would produce an infinite family. This is the configuration underlying the value $t=5/2$ in the present weighting scheme and matches the lower bound construction.

\section*{Acknowledgements}

Levente Bodn\'ar, Oleg Pikhurko and Shumin Sun were supported by ERC Advanced Grant 101020255. Oleg Pikhurko was also supported 
in part by a grant of access to OpenAI models through the ChatGPT for Academic Researchers program.
Yan Wang was supported by the National Key R\&D Program of China under Grant No.~2022YFA1006400 and the National Natural Science Foundation of China under Grant No.~12571376. Part of this research was done during Yan Wang's visit to Warwick that was supported by the SJTU--Warwick Joint Seed Fund.

\paragraph{Declaration of generative AI use.}
During the development of the lower-bound construction, the authors used ChatGPT to assist with the finite case analysis and the associated computations. In particular, the tool was used to help organise and examine the relevant cases, and to generate a Python program that computes $\nu(c,h-c)$ for the relevant values of $c$ and $3\le h\le 9$. Generative AI was also used to write some routine parts of the SageMath code for the upper bound. The authors  take full responsibility for the content of the manuscript.

\bibliographystyle{abbrv}
\bibliography{ref}

\appendix

\section{Computational details}\label{se:appendix-computation}

This appendix describes the exhaustive generation of the finite cluster families used in the upper-bound proofs and the finite linear programs obtained from Constraints~\eqref{lp:F}--\eqref{lp:S}. The full implementation and exact certificate data are provided separately.

\subsection{Generating $\ExtIntClusters{3}k$}\label{subsec:appendix-enumeration}

The program generates a family $\ExtIntClusters{3}k$ beginning with a single edge and closing under two operations.

\begin{enumerate}[label=(\roman*)]
	\item \emph{External gluing.} Two already generated $3$-graphs are taken as disjoint copies and a selected pair of the first is identified with a selected pair of the second. Thus the two copies initially meet exactly in the gluing pair. The selected pairs must satisfy a rule in $\C R$, and the union is retained only if it is $\cG{3}{k}$-free. Both orientations of the pair are considered.
	
	\item \emph{Internal gluing.} If two vertices $u,v$ of a generated graph satisfy $\CI_F^{<k}(uv)=\{0\}$, they may be identified, provided that the resulting simple $3$-graph remains $\cG{3}{k}$-free.
\end{enumerate}

Recall that two edge-disjoint subgraphs $F,H$ are mergeable via the pair $uv$ if there is a rule $\Merge{A}{B} \in \C R$ such that $A \subseteq \CI_F(uv)$ and $B \subseteq \CI_H(uv)$, or vice versa. In the partition, these subgraphs are replaced by their union $F \cup H$. We store only one representative of each isomorphism type during generation. Since $F$ and $H$ may share vertices beyond the merging pair, an abstract merging need not be isomorphic to an external gluing of their stored representatives. We account for this by first gluing disjoint copies externally along the merging pair and then identifying the remaining common vertices internally.

An external gluing of two $\cG{3}{k}$-free graphs sharing exactly two vertices cannot create an $\ell^-$-configuration using edges from both sides with $\ell < k$. Indeed, if such a configuration uses $a$ edges from one side and $b$ from the other, then the two portions use at least $a+2$ and $b+2$ vertices, respectively, and their union uses at least $a+b+2$ vertices. Consequently, for an external gluing it is enough to test whether a $k$-configuration is created, which can be easily guaranteed by asserting that no pair $f_u, f_v$ and $h_u, h_v$ is glued with $k \in \CI_F^{<k}(f_u f_v) + \CI_H^{<k}(h_u h_v)$.

A similar conclusion holds for an internal gluing: any newly created smaller $\ell^-$-configuration would lift to an $\ell$-claim of the identified pair before the identification, hence internal gluings are only tested for pairs $uv \in \binom{V(F)}2$ with $\CI_F^{<k}(uv) = \{0\}$.

We record a simple observation which is used repeatedly by the implementation. Let $F$ be a $\cG{3}{k}$-free $3$-graph. If $1\le i<k$ and $F$ $i$-claims a pair $xy$, then the $i$ witnessing edges together with $xy$ use exactly $i+2$ vertices: fewer vertices would give an $i^-$-configuration. Moreover, the induced subgraph on these $i+2$ vertices contains exactly $i$ edges, since an additional edge would give either an $(i+1)^-$-configuration or, when $i+1=k$, a $k$-configuration. We call such a witness an \emph{$i$-tight set}. Thus, for every graph occurring in the enumeration, all truncated claim sets can be recovered from its tight sets. These tight sets accelerate the generation. The tight sets of an externally glued pair can be calculated quickly from the tight sets of the components.

\lstdefinestyle{pseudocode}{
  basicstyle=\ttfamily\small,
  frame=single,
  numbers=none,
  columns=fullflexible,
  keepspaces=true,
  showstringspaces=false,
  breaklines=true,
  xleftmargin=1em,
  framexleftmargin=.5em
}

\noindent\begin{minipage}{\linewidth}
\begin{lstlisting}[style=pseudocode,caption={Enumeration of $\ExtIntClusters{3}k$.}]
Input: k and the rules Rgen
Output: a family C of cluster isomorphism types

C <- {one-edge graph}
Q <- queue containing the one-edge graph

while Q is not empty:
    remove F from Q

    for every H in C and every rule rho in Rgen:
        for every eligible pair fu,fv of F and eligible ordered pair hu,hv of H:
            J <- F sqcup H / (fu=hu, fv=hv)
            if J is admissible and its isomorphism type is new:
                add J to C and Q

    for every pair u,v of F with truncated claim set {0}:
        J <- F/(u=v)
        if J is admissible and its isomorphism type is new:
            add J to C and Q

return C
\end{lstlisting}
\end{minipage}

For comparing 3-graphs, the algorithm uses Nauty \cite{nauty}.

\subsection{The linear program and data}\label{subsec:appendix-lp}


After the reduction in Remark~\ref{rem:finite-lp}, all retained claim sets have bounded multiplicity. Since all weights are nonnegative, it is enough to retain admissible multiplicity vectors which are maximal coordinatewise. Another set of linear constraints comes from each generated graph $F \in \ExtIntClusters{3}k$.
The LP is solved over the rational field using the PPL backend in SageMath \cite{sage}, so all reported objective values and weights are exact rationals.

The computation gives the following data.
\begin{center}
\begin{tabular}{lcc}
 & $k=8$ & $k=9$\\
\midrule
Found optimal $t$ & $\dfrac{13272}{5053}$ & $\dfrac52$\\[4mm]
$\left| \ExtIntClusters{3}k \right|$ & $5683$ & $9224$\\[1mm]
$\max_{F \in \ExtIntClusters{3}k} v(F)$ & $21$ & $16$\\[1mm]
$\# F$ with $w(F)=t|F|$ & $22$ & $50$\\[1mm]
$\# S$ for which $S^{(k)}$ lacks the $\C R$-property & $116$ & $224$\\

\bottomrule
\end{tabular}
\end{center}

The complete code, exact weight vectors and the \emph{tight} $3$-graphs (that is, $F\in\ExtIntClusters{3}{k}$ with $w(F)=t|F|$) can be found in the ancillary folder on arXiv. 

\end{document}